\documentclass[twoside, 12pt]{article}

\usepackage[russian, english]{babel}
\usepackage{epsfig}
\usepackage{amssymb,amsmath,amsfonts,soul,amsthm,enumerate}
\usepackage{amsfonts}
\usepackage{amsmath}
\usepackage{graphicx}
\usepackage{amsthm}
\usepackage[cp1251]{inputenc}
\usepackage[T2A]{fontenc}
\usepackage{mathrsfs}
\usepackage[english]{babel}
\newtheorem{theorem}{Theorem}[section]
\newtheorem{lemma}{Lemma}[section]

\newtheorem{corollary}{Corollary}[section]

\numberwithin{equation}{section}
 \def\@evenhead{\vbox{\hbox to \textwidth{\thepage\hfil\sl\leftmark\strut}\hrule}}
 \def\@oddhead{\vbox{\hbox to \textwidth{\rightmark\hfill\thepage\strut}\hrule}}

\usepackage{hyperref}
\usepackage{alphalph}
\usepackage[subnum]{cases}
\usepackage{array}

\newcommand{\sign}{\text{sign}}

\title{\bf  Similar transformation of one class of correct restrictions}
\author{B.N. Biyarov}

\begin{document}
 \sloppy
\maketitle

\markboth{\hfill{\footnotesize\rm   B.N.~Biyarov}\hfill}
{\hfill{\footnotesize\sl   Similar transformation of one class of correct restrictions}\hfill}
\vskip 0.3cm  
\vskip 0.7 cm

\noindent {\bf Key words:}  maximal (minimal) operator, correct restriction, correct extension, similar operators, singular coefficients, Riesz basis with brackets.

\vskip 0.2cm

\noindent {\bf AMS Mathematics Subject Classification:}  47A05,  47A10.
\vskip 0.2cm

%
\noindent {\bf Abstract.} The description of all correct restrictions of the maximal operator are considered in a Hilbert space. A class of correct restrictions are obtained for which a similar transformation has the domain of the fixed correct restriction. The resulting theorem is applied to the study of n-order differentiation operator with singular coefficients.



\section{\large Introduction}
\label{sec:1} 
Let a closed linear operator $L$ be given in a Hilbert space $H$. The linear equation
\begin{equation}\label{eq:1.1}
Lu=f
\end{equation}
is said to be \textit{correctly solvable} on $R(L)$ if $\|u\|\leq C\|Lu\|$ for all $u\in D(L)$ (where $C>0$ does not depend on $u$) and  \textit{everywhere solvable} if $R(L)=H$. If \eqref{eq:1.1} is simultaneously correct and solvable everywhere, then we say that $L$ is a \textit{correct operator}. A correctly solvable operator $L_0$ is said to be \textit{minimal} if $R(L_0)\neq H$. A closed operator $\widehat{L}$ is called a \textit{maximal operator} if $R(\widehat{L})=H$ and Ker$\,\widehat{L}\neq \{0\}$. An operator $A$ is called a \textit{restriction} of an operator $B$ and $B$ is said to be an  \textit{extension} of $A$ if $D(A)\subset D(B)$ and $Au=Bu$ for all $u\in D(A)$.

Note that if one of the correct restriction $L$ of a maximal operator $\widehat{L}$ is known, then the inverses of all correct restrictions of $\widehat{L}$ have in the form \cite{Kokebaev}
\begin{equation}\label{eq:1.2}
L_K^{-1}f=L^{-1}f+Kf,
\end{equation}
where $K$ is an arbitrary bounded linear operator in $H$ such that $R(K)\subset$ Ker$\,\widehat{L}$.

Let $L_0$ be some minimal operator, and let $M_0$ be another minimal operator related to $L_0$ by the equation $(L_0u, v)=(u, M_0v)$ for all $u\in D(L_0)$ and $v\in D(M_0)$. Then $\widehat{L}=M_0^*$ and $\widehat{M}=L_0^*$ are maximal operators such that $L_0\subset \widehat{L}$ and $M_0\subset \widehat{M}$. 
A correct restriction $L$ of a maximal operator $\widehat{L}$ such that $L$ is simultaneously a  correct extension of the minimal operator $L_0$ is called a \textit{boundary correct extension}.
The existence of at least one boundary correct extension $L$ was proved by Vishik in {\cite{Vishik}}, that is, $L_0\subset L\subset \widehat{L}$.

The inverse operators to all possible correct restrictions $L_K$ of the maximal operator $\widehat{L}$ have the form \eqref{eq:1.2}, then 
$$
D(L_K)=\big\{u\in D(\widehat{L}):\, (I-K\widehat{L})u\in D(L)\big\}
$$ 
is dense in $H$ if and only if $\mbox{Ker}\, (I+K^*L^*)=\{0\}$.
 All possible correct extensions $M_K$ of $M_0$ have inverses of the form
$$
M_K^{-1}f=(L_K^*)^{-1}f=(L^*)^{-1}f+K^*f,
$$
where $K$ is an arbitrary bounded linear operator in $H$ with $R(K)\subset \mbox{Ker}\, \widehat{L}$ such that 
\[\mbox{Ker}\, (I+K^*L^*)=\{0\}.\]

The main result of this work is the following
\begin{theorem}\label{th:1.1}
Let $L$ be boundary correct extension of $L_0$, that is, $L_0 \subset L\subset \widehat{L}$. If $L_K$ is densely defined in $H$ and
$$
R(K^*)\subset D(L^*)\cap D(L_K^*),
$$
where $K$ and $L$ are the operators in representation \eqref{eq:1.2}, then $\overline{KL}_K$ is bounded in $H$ and a correct restriction $L_K$ of the maximal operator $\widehat{L}$ is similar to the correct operator  
$$
A_K=L-\overline{KL}_KL \;\; \mbox{on} \;\; D(A_K)=D(L).
$$
\end{theorem}


\section{\large Preliminaries}
\label{sec:2}
In this section, we present some results for correct restrictions and extensions {\cite{Biyarov3}} which are used in Section \ref{sec:3}.

Let $A$ and $B$ be bounded operators in a Hilbert space $H$. An operators $A$ and $B$ are said to be \textit{similar} if there exist an invertible operator $P$ such that $P^{-1}AP=B$. Similar operators have the same spectrum. If at least one of two operators $A$ and $B$ is invertible, then the operators $AB$ and $BA$ are similar.

\begin{lemma}\label{lem:2.1}
Let $L$ be a densely defined correct restriction of the maximal operator $\widehat{L}$ in a Hilbert space $H$. Then the operator $KL$ is bounded on $D(L)$ (that is, $\overline{KL}$ is bounded in  $H$)  if and only if 
\[R(K^*)\subset D(L^*).\]
\end{lemma}

\begin{proof}
Let $R(K^*)\subset D(L^*)$. Then, by virtue of  $(KL)^*=L^*K^*$, we have that $\overline{KL}$ is bounded in $H$, where $\overline{KL}$ is the closure of the operator $KL$ in $H$. Here we have used the boundedness of the operator $L^*K^*$. Then the operator $KL$ is bounded on $D(L)$.
Conversely, let $KL$ be bounded on $D(L)$. Then $\overline{KL}$ is bounded on $H$, by virtue of $(KL)^*=(\overline{KL})^*$ and that $(KL)^*$ is defined on the whole space $H$. Then the operator $K^*$ transfers any element $f$ in $H$ to $D(L^*)$. Indeed, for any element $g$ of $D(L)$ we have
\[(Lg, K^*f)=(KLg, f)=(g, (KL)^*f).\]
Therefore, $K^*f$  belongs to the domain $D(L^*)$. 
\end{proof}

\begin{lemma}\label{lem:2.2}
Let $L_K$ be a densely defined correct restriction of the maximal operator $\widehat{L}$ in a Hilbert space $H$. Then $D(L^*)= D(L_K^*)$ if and only if $R(K^*)\subset D(L^*)\cap D(L_K^*)$, where $L$ and $K$ are the operators from  representation \eqref{eq:1.2}.
\end{lemma}

\begin{proof}
If $D(L^*)= D(L_K^*)$, then from  representation \eqref{eq:1.2} we easily get 
\[R(K^*)\subset D(L^*)\cap D(L_K^*)=D(L^*)= D(L_K^*)\]
Let us prove the converse. If
\[R(K^*)\subset D(L^*)\cap D(L_K^*),\]
then we obtain
\begin{equation}\label{eq:2.2}
(L_K^*)^{-1}f=(L^*)^{-1}f+K^*f=(L^*)^{-1}(I+L^*K^*)f,
\end{equation}
\begin{equation}\label{eq:2.3}
(L^*)^{-1}f=(L_K^*)^{-1}f-K^*f=(L_K^*)^{-1}(I-L_K^*K^*)f,
\end{equation}
for all $f$ in $H$. 
It follows from \eqref{eq:2.2} that $D(L_K^*)\subset D(L^*)$, and taking into account  \eqref{eq:2.3} this implies that $D(L^*)\subset D(L_K^*)$. Thus $D(L^*)= D(L_K^*)$. 
\end{proof}

\begin{corollary}\label{cor:2.3}
Let $L_K$ be any densely defined correct restriction of the maximal operator $\widehat{L}$ in a Hilbert space $H$. If $R(K^*)\subset D(L^*)$ and $\overline{KL}$ is a compact operator in $H$, then 
\[D(L^*)=D(L_K^*).\]
\end{corollary}

\begin{proof}
Compactness of $\overline{KL}$ implies compactness of $L^*K^*$. Then $R(I+L^*K^*)$  is a closed subspace in $H$. It follows from the dense definiteness of $L_K$  that $R(I+L^*K^*)$ is a dense set in $H$. Hence $R(I+L^*K^*)=H$. Then from  equality \eqref{eq:2.2} we get $D(L^*)=D(L_K^*)$. 
\end{proof}

\begin{lemma}\label{lem:2.3}
If $R(K^*)\subset D(L^*)\cap D(L_K^*)$, then bounded operators $I+L^*K^*$ and $I-L_K^*K^*$ from \eqref{eq:2.2} and \eqref{eq:2.3}, respectively, have a bounded inverse defined on $H$.
\end{lemma}

\begin{proof}
By virtue of the density of the domains of the operators $L_K^*$ and $L^*$ it follows that the operators $I+L^*K^*$ and $I-L_K^*K^*$ are invertible. Since from \eqref{eq:2.2} and \eqref{eq:2.3} we have 
$\mbox{Ker}\, (I+L^*K^*)=\{0\}$ and $\mbox{Ker}\, (I-L_K^*K^*)=\{0\}$, respectively.
From  representations \eqref{eq:2.2} and \eqref{eq:2.3} we also note that $R(I+L^*K^*)=H$ and $R(I-L_K^*K^*)=H$, since $D(L^*)= D(L_K^*)$.
The inverse operators $(I+L^*K^*)^{-1}$ and $(I-L_K^*K^*)^{-1}$ of the closed operators $I-L_K^*K^*$
 and $I+L^*K^*$, respectively, are closed. Then the closed operators $(I+L^*K^*)^{-1}$ and $(I-L_K^*K^*)^{-1}$, defined on the whole of $H$, are bounded.
\end{proof}

Under the assumptions of Lemma \ref{lem:2.3} the operators $KL$ and $KL_K$ will be (see {\cite{Biyarov2}}) 
 restrictions of the bounded operators $\overline{KL}$ and $\overline{KL}_K$, respectively, where the bar denotes the closure of operators in $H$. Thus $(I-L_K^*K^*)^{-1}=I+L^*K^*$ and $(I-\overline{KL}_K)^{-1}=I+\overline{KL}$.

In what follows, we need the following theorem
\begin{theorem}[Theorem 1.1 {\cite[p.\,307]{Gohberg}}]\label{ThGohberg}
The sequence $\{\psi_j\}_{j=1}^{\infty}$ biorthogonal to a basis $\{\phi_j\}_{j=1}^{\infty}$ of a Hilbert space $H$ is also a basis of $H$.
\end{theorem}


\section{\large Proof of Theorem \ref{th:1.1}}
\label{sec:3}
In this section we prove our main result Theorem \ref{th:1.1}.
\begin{proof}
We transform \eqref{eq:1.2} to the form
\begin{equation}\label{eq:3.1}
L_K^{-1}=L^{-1}+K=(I+KL)L^{-1}.
\end{equation}
By Lemma \ref{lem:2.1} and Lemma \ref{lem:2.3} the operators $KL$ and $KL_K$ are bounded and $I+KL$ is invertible with 
$$
(I+KL)^{-1}=I-KL_K.
$$
Then we have
\begin{eqnarray*} 
A_K^{-1}& =& (I+KL)^{-1}L_K^{-1}(I+KL)  \\[6pt]
& =& (I+KL)^{-1}(I+KL)L^{-1}(I+KL)=L^{-1}(I+KL).
\end{eqnarray*}
Hence, by Corollary 1 \cite[p.\,259]{Fillmore} we have $D(A_K)=D(L)$ and
$$
A_K=(I-KL_K)L=L-KL_K L.
$$
\end{proof}

\begin{corollary}\label{cor:3.1}
Suppose the hypothesis of Theorem \ref{th:1.1} is satisfied. Then a correct extension $L_K^*$ of a minimal operator $M_0$ is similar to the correct operator
$$
A_K^*=L^*(I-L_K^*K^*),
$$
on
$$
D(A_K^*)=\big\{v\in H:\, (I-L_K^*K^*)v\in D(L^*)\big\}.
$$
\end{corollary}

\section{\large An application of Theorem \ref{th:1.1} to n-order differentiation operator}
\label{sec:4}
In this section, we give some applications of the main result to differential operators.
 
As a maximal operator $\widehat{L}$ in $L^2(0, 1)$, we consider 
$$
\widehat{L}y = y^{n},
$$ 
with domain $D(\widehat{L}) = W_2^n(0, 1), \,\; n\in \mathbb{N}$. Then the minimal operator $L_0$ is the restriction of $\widehat{L}$ on $D(L_0) = \mathring{W}_2^n(0, 1)$. As a fixed boundary correct extension $L$ of the minimal operator $L_0$, we take the restriction of $\widehat{L}$ on
$$
D(L)=\big\{y\in W_2^n(0, 1):\, y^{(i)}(0)+y^{(i)}(1)=0, \;\, i=1, 2, \ldots, n-1\big\}.
$$ 
We find the inverse to all correct restrictions of $L_K\subset \widehat{L}$
$$
L_K^{-1}=L^{-1}+K,
$$
where
$$
Kf=\sum_{i=1}^{n}w_i(x)\int_0^1f(t)\overline{\sigma}_i(t)\, dt, \;\; \sigma_i\in L^2(0, 1),
$$
and $w_i\in \mbox{Ker}\, \widehat{L}, \; i=1, 2, \ldots, n$ are linearly independent functions with the properties
\[
w_i^{(k-1)}=\left\{
\begin{array}{lr}
1, \;\;  i=k, \\[2pt]
0, \;\;  i\neq k, 
\end{array}
\right. \quad i, k=1, 2, \ldots, n.
\]
Then the operator $L_K$ is the restriction of $\widehat{L}$ on 
$$
D(L_K)=\big\{u\in W_2^n(0, 1):\, u^{(k-1)}(0)+u^{(k-1)}(1)=\int_0^1u^{(n)}(t)\overline{\sigma}_k(t)\, dt, \;\, k=1, 2, \ldots, n\big\}.
$$
We will consider restrictions of $L_K$ with dense domains in $L^2(0, 1)$, that is, 
$$
\overline{D(L_K)}=L^2(0, 1).
$$ 
If $R(K^*)\subset D(L^*)$, then by Corollary \ref{cor:3.1} the operators $\overline{KL}$ and $\overline{KL}_K$ will be bounded in $L^2(0, 1)$ (where bar denotes closure).
Since $\overline{KL}$ is a compact operator, then by Lemma \ref{lem:2.3} the operator $I+KL$ is invertible and
$(I+KL)^{-1}=I-KL_K$. The operator $\overline{KL}$ is bounded if and only if
$$
\sigma_i\in D(L^*)=\big\{\sigma_i\in W_2^n(0, 1):\, \sigma_i^{(k-1)}(0)+\sigma_i^{(k-1)}(1)=0,  \,\;\, i, k=1, 2, \ldots, n\big\}.
$$
Hence, we have
$$
KLy=\sum_{i=1}^{n}w_i(x)\int_0^1 y^{(n)}(t)\overline{\sigma}_i(t)\, dt= (-1)^n\sum_{i=1}^{n}w_i(x)\int_0^1 y(t)\overline{\sigma}_i^{(n)} (t)\, dt.
$$
We find the operator $KL_K$. For this, we invert the operator 
$$
(I+KL)y=y+(-1)^n\sum_{i=1}^{n}w_i(x)\int_0^1 y(t)\overline{\sigma}_i^{(n)}(t)\, dt=u, 
$$
where $y\in D(L), \;\, u\in D(L_K)$. Then we can write 
$$
y=(I-KL_K)u=u-(-1)^n \sum_{i=1}^{n}w_i(x)\sum_{j=1}^{n}\beta_{ij} \int_0^1 u(t)\overline{\sigma}_j^{(n)}(t)\, dt, 
$$
where $\beta_{ij}, \;\, i, j=1, 2, \ldots, n$ is an elements of the inverse matrix $U^{-1}$ of a matrix $U$.
\begin{equation*}
U = 
\begin{pmatrix}
1+(-1)^{n-1}\overline{\sigma}_1^{(n-1)}(0) & (-1)^{n-2}\overline{\sigma}_1^{(n-2)}(0) & \cdots & \overline{\sigma}_1^{(2)}(0) & -\overline{\sigma}_1^{(1)}(0)  & \overline{\sigma}_1(0) \\
(-1)^{n-1}\overline{\sigma}_2^{(n-1)}(0) & 1+(-1)^{n-2}\overline{\sigma}_2^{(n-2)}(0) & \cdots & \overline{\sigma}_2^{(2)}(0)  & -\overline{\sigma}_2^{(1)}(0) & \overline{\sigma}_2(0) \\
\vdots  & \vdots  & \ddots & \vdots  & \vdots & \vdots \\
(-1)^{n-1}\overline{\sigma}_n^{(n-1)}(0) & (-1)^{n-2}\overline{\sigma}_n^{(n-2)}(0) & \cdots & \overline{\sigma}_n^{(2)}(0)  & -\overline{\sigma}_n^{(1)}(0) & 1+\overline{\sigma}_n(0)
\end{pmatrix}
\end{equation*}
Note that the conditions $R(K^*)\subset D(L^*)$ and $\overline{D(L_K)}=L^2(0, 1)$ imply that $\det U\neq0$.
Thereby, the operator
$$
KL_Ku=(-1)^n \sum_{i=1}^{n}w_i(x)\sum_{j=1}^{n}\beta_{ij} \int_0^1 u(t)\overline{\sigma}_j^{(n)}(t)\, dt, 
$$
is a bounded operator in $L^2(0, 1)$. Then the operator $A_K$ has the form
$$
A_Kv=Lv-\overline{KL}_KLv=v^{(n)}-(-1)^n \sum_{i=1}^{n}w_i(x)\sum_{j=1}^{n}\beta_{ij} \int_0^1 v^{(n)}(t)\overline{\sigma}_j^{(n)}(t)\, dt, 
$$
on
$$
D(A_K)=D(L)=\big\{v\in W_2^n(0, 1):\, v^{(k-1)}(0)+v^{(k-1)}(1)=0,  \,\;\,  k=1, 2, \ldots, n\big\}.
$$
The operator $A_K$ can be written as
$$
A_Kv=v^{(n)}-(-1)^n \sum_{i=1}^{n}w_i(x)\sum_{j=1}^{n}\beta_{ij} F_j(v),
$$
where
$$
F_j(v)=<F_j, v>= \int_0^1 v^{(n)}(t)\overline{\sigma}_j^{(n)}(t)\, dt, \;\; j=1, 2, \ldots, n.
$$
It can be seen that $F_j\in W_2^{-n}(0, 1)$ in the sense of Lions-Margins (see \cite{Lions}). 

We transform the boundary conditions of $L_K$ to the form
\begin{equation*}
U 
\begin{pmatrix}
& u(0)+u(1) \\
& u^{(1)}(0)+u^{(1)}(1) \\
&\vdots \\
& u^{(n-1)}(0)+u^{(n-1)}(1)
\end{pmatrix}
=
\begin{pmatrix}
& \int_0^1 u(t)\overline{\sigma}_1^{(n)}(t)\, dt \\
& \int_0^1 u(t)\overline{\sigma}_2^{(n)}(t)\, dt \\
&\vdots \\
& \int_0^1 u(t)\overline{\sigma}_n^{(n)}(t)\, dt
\end{pmatrix}
\end{equation*}
Then we get
\begin{equation}\label{eq:4.2}
u^{(i-1)}(0)+u^{(i-1)}(1)=\sum_{j=1}^{n}\beta_{ij} \int_0^1 u(t)\overline{\sigma}_j^{(n)}(t)\, dt, \;\; i=1, 2, \ldots, n,
\end{equation}
where $u\in D(L_K), \;\; \sigma_j^{(n)}\in L^2(0, 1), \; j=1, 2, \ldots, n$.
The boundary condition \eqref{eq:4.2} is regular in Shkalikov sense (see \cite{Shkalikov}). Then, by virtue of \cite{Shkalikov}, the operator $L_K$ has a system of root vectors forming a Riesz basis with brackets in $L^2(0, 1)$. Thereby the operator $A_K$, being similar to the operator $L_K$, also has a basis with brackets property. The eigenvalues of these operators coincide. If $\{u_k\}_1^{\infty}$ are eigenfunctions of the operator $L_K$, then the eigenfunctions $v_k$ of the operator $A_K$ are related to them by the relations
$$
u_k=(I+KL)v_k=v_k+(-1)^n \sum_{i=1}^{n}w_i(x)\int_0^1 v_k(t)\overline{\sigma}_i^{(n)}(t)\, dt, \;\; k=1, 2, \ldots, n.
$$
If, in particular, we take
$$
\sigma_i^{(n)}(x)=\sign(x-x_i), \;\; 0<x_i<1, \;\; i=1, 2, \ldots, n,
$$
then we get
$$
F_i(v)=-2v^{(n-1)}(x_i), \;\; i=1, 2, \ldots, n.
$$
By Corollary \ref{cor:3.1}, Theorem \ref{ThGohberg}, and \cite[p.\,928]{Shkalikov2}, we can assert that the system of root vectors of the adjoint operator
$$
A_K^*v=(-1)^n\frac{d^n}{dx^n}\bigg[v(x)-(-1)^n \sum_{i, j=1}^{n}\overline{\beta}_{ij}\sigma_j^{(n)}(x) \int_0^1 v(t)w_i(t)\, dt\bigg],
$$
on
$$
D(A_K^*)=\bigg\{v\in L^2(0, 1):\, v(x)-(-1)^n \sum_{i, j=1}^{n}\overline{\beta}_{ij}\sigma_j^{(n)}(x) \int_0^1 v(t)w_i(t)\, dt\in D(L)\bigg\},
$$
form a Riesz basis with brackets in $L^2(0, 1)$.


\section{\large Example in case $n = 2$}
\label{sec:5}

If the maximal operator $\widehat{L}$ acts as 
$$
\widehat{L}y=-y''
$$ 
on the domain $D(\widehat{L})=W_2^2(0, 1)$, then the minimal operator $L_0$ is the restriction of $\widehat{L}$ on $D(L_0)=\mathring{W}_2^2(0, 1)$. As a fixed operator $L$ we take the restriction of $\widehat{L}$ on
$$
D(L)=\big\{y\in W_2^2(0, 1):\, y(0)=y(1)=0\big\}.
$$
Then
\begin{align*} 
L_K^{-1}f=L^{-1}f+Kf&=-\int_0^x (x-t)f(t)\,dt+x\int_0^1 (1-t)f(t)\,dt\\
& \quad +(1-x) \int_0^1f(t)\overline{\sigma}_1(t)\,dt+x\int_0^1f(t)\overline{\sigma}_2(t)\,dt,\\[5pt]
Kf&=(1-x) \int_0^1f(t)\overline{\sigma}_1(t)\,dt+x\int_0^1f(t)\overline{\sigma}_2(t)\,dt.
\end{align*}
$KL$ is bounded in $L^2(0, 1)$, if $R(K^*)\subset D(L^*)=D(L)$, that is, 
$$
\sigma_1,\, \sigma_2\in D(L)=\big\{\sigma_1, \sigma_2\in W_2^2(0, 1):\, \sigma_1(0)=\sigma_1(1)=\sigma_2(0)=\sigma_2(1)=0\big\},
$$ 
and has the form
$$
KLy=-(1-x) \int_0^1y(t)\overline{\sigma}_1''(t)\,dt-x\int_0^1y(t)\overline{\sigma}_2''(t)\,dt.
$$
The operator $KL_K$ is also bounded in $L^2(0, 1)$ and 
\begin{align*} 
KL_Ku & = -\frac{1-x}{\Delta} \Big[\big(1- \overline{\sigma}_2'(1)\big)\int_0^1u(t)\overline{\sigma}_1''(t)\,dt +  \overline{\sigma}_1'(1)\int_0^1u(t)\overline{\sigma}_2''(t)\,dt \Big]\\[5pt]
&\quad\; +\frac{x}{\Delta} \Big[\big(1+ \overline{\sigma}_1'(0)\big)\int_0^1u(t)\overline{\sigma}_2''(t)\,dt -  \overline{\sigma}_2'(0)\int_0^1u(t)\overline{\sigma}_1''(t)\,dt \Big], 
\end{align*}
where
$$
\Delta=\big(1+ \overline{\sigma}_1'(0)\big)\big(1- \overline{\sigma}_2'(1)\big)+ \overline{\sigma}_2'(0) \,\overline{\sigma}_1'(1).
$$
Then the operator $A_K$ has the form
\begin{align*} 
A_Kv & = -v''-\frac{1}{\Delta} \Big[\big((1-x)(1- \overline{\sigma}_2'(1))-x \overline{\sigma}_2'(0) \big)\int_0^1v''(t)\overline{\sigma}_1''(t)\,dt \\[5pt]
&\quad\; +  \big((1-x)\overline{\sigma}_1'(1)+x(1+ \overline{\sigma}_1'(0)) \big) \int_0^1v''(t)\overline{\sigma}_2''(t)\,dt \Big],
\end{align*}
on
$$
D(A_K)= D(L)=\big\{v\in W_2^2(0, 1):\, v(0)=v(1)=0\big\},
$$ 
where  $\sigma_1'',\, \sigma_2''\in L^2(0, 1)$.

We rewrite the operator $A_K$ in the form
\begin{equation}\label{eq:5.1}
A_Kv=-v''+a(x)F_1(v)+b(x)F_2(v),
\end{equation}
where
\begin{align*} 
& a(x)= -\frac{1}{\Delta} \big((1-x)(1- \overline{\sigma}_2'(1))-x \overline{\sigma}_2'(0) \big), \;\; F_1(v)=\int_0^1v''(t)\overline{\sigma}_1''(t)\,dt,  \\[3pt]
& b(x)=-\frac{1}{\Delta}\big((1-x)\overline{\sigma}_1'(1)+x(1+ \overline{\sigma}_1'(0)), \;\; F_2(v)=\big) \int_0^1v''(t)\overline{\sigma}_2''(t)\,dt.
\end{align*}
Note that $F_1, F_2\in W_2^{-2}(0, 1)$ in the sense of Lions-Margins (see \cite{Lions}). 

Further, we see that the operator $L_K$ acts as $\widehat{L}$ on the domain
\begin{eqnarray*} 
D(L_K)=\Bigg\{u\in W_2^2(0, 1): \qquad\qquad\qquad\qquad\qquad\qquad\qquad\qquad\qquad\qquad
\end{eqnarray*}
\begin{eqnarray*} 
\begin{pmatrix}
& 1+\overline{\sigma}_1'(0)& 0 & -\overline{\sigma}_1'(1) &0 \\
& \overline{\sigma}_2'(0) & 0 & 1-\overline{\sigma}_2'(1) &0
\end{pmatrix}
\begin{pmatrix}
& u(0) \\
& u'(0) \\
& u(1) \\
& u'(1)
\end{pmatrix}
=
\begin{pmatrix}
& -\int_0^1 u(t)\overline{\sigma}_1''(t)\,dt \\
& -\int_0^1 u(t)\overline{\sigma}_2''(t)\,dt 
\end{pmatrix}
\Bigg\},
\end{eqnarray*}
and
$$
J_{13}=\big(1+\overline{\sigma}_1'(0)\big)\big(1-\overline{\sigma}_2'(1)\big)+\overline{\sigma}_2'(0)\,\overline{\sigma}_1'(1)=\Delta\neq0,
$$
since $R(K^*)\subset D(L^*)$ and $\overline{D(L_K)}=L^2(0, 1)$ .
Then the left hand of this boundary condition is non-degenerate according to Marchenko \cite{Marchenko}, hence regularly according to Birkhoff (see \cite{Shkalikov}). By virtue of Theorem (see \cite[p.\,15]{Shkalikov}), the system of root vectors of the operator $L_K$ form a Riesz basis with brackets in $L^2(0, 1)$.
Thus, by virtue of Theorem \eqref{eq:1.1} the system of root vectors of $A_K$  also form a Riesz basis with brackets and the eigenvalues of $L_K$ and $A_K$ coincide, and the eigenfunctions are related to each other as follows
$$
u_k=v_k-(1-x)\int_0^1v_k(t)\overline{\sigma}_1''(t)\,dt -x\int_0^1v_k(t)\overline{\sigma}_2''(t)\,dt, \;\, k\in \mathbb{N}.
$$
If in the particular case we take 
\begin{equation}\label{eq:5.2}
\begin{split}
& \overline{\sigma}_1''(x)=\sign(x-x_1)-\sign(x-x_2), \\[3pt]
&  \overline{\sigma}_2''(x)=x\big[\sign(x-x_1)-\sign(x-x_2)\big],
\end{split}
\end{equation}
where $0<x_1<x_2<1$, then we get 
\begin{align*} 
& F_1(v)=2v'(x_2)-2v'(x_1), \\[3pt]
& F_2(v)=2x_2v'(x_2)-2x_1v'(x_1)-2v(x_2)+2v(x_1),
\end{align*}
in \eqref{eq:5.1}.

In the case $n=2$, by Corollary \ref{cor:3.1}, Theorem \ref{ThGohberg}, and \cite[p.\,928]{Shkalikov2}, we can assert that the system of root vectors of the operator 
$$
A_K^*v=(-1)^2\frac{d^2}{dx^2}\Big[v(x)-c(x) \int_0^1 (1-t)v(t)\, dt -d(x) \int_0^1 tv(t)\, dt \Big],
$$
on
$$
D(A_K^*)=\Big\{v\in L^2(0, 1):\, v(x)-c(x) \int_0^1 (1-t)v(t)\, dt -d(x) \int_0^1 tv(t)\, dt\in D(L)\Big\},
$$
form a Riesz basis with brackets in $L^2(0, 1)$, where
\begin{align*} 
& c(x)=-\frac{1}{\Delta}\Big[\big(1-\sigma_2'(1)\big)\sigma_1''(x)+\sigma_1'(1)\sigma_2''(x)\Big], \\[3pt]
& d(x)= \frac{1}{\Delta}\Big[\big(1+\sigma_1'(0)\big)\sigma_2''(x)-\sigma_2'(0)\sigma_1''(x)\Big].
\end{align*}
Note that
$$
\sigma_1'', \sigma_2''\in L^2(0, 1), \;\; D(L)=\big\{y\in W_2^2(0, 1):\, y(0)=y(1)=0\big\}.
$$
For clarity, we consider the special case \eqref{eq:5.2}, then we have
\begin{align*} 
& c(x)=\frac{\sign(x-x_1)-\sign(x-x_2)}{\Delta}\Big[1+\frac{x_2^3-x_1^3}{3}-\frac{x_2^2-x_1^2}{2}x\Big], \\[3pt]
& d(x)= \frac{\sign(x-x_1)-\sign(x-x_2)}{\Delta}\Big[\big(1+x_2-x_1 -\frac{x_2^2-x_1^2}{2}\big)x-\frac{x_2^2-x_1^2}{2}+\frac{x_2^3-x_1^3}{3}\Big].
\end{align*}
The domain of $A_K^*$ will have the form
\begin{eqnarray*} 
&& D(A_K^*)= \bigg\{v\in L^2(0, 1)\cap W_2^2(0, x_1)\cap W_2^2(x_1, x_2)\cap W_2^2(x_2, 1):\, v(0)=v(1)=0, \; \\
&&\quad v(x_1-0)-v(x_1+0)=-c(x_1+0)\int_0^1(1-t)v(t)\,dt-d(x_1+0)\int_0^1tv(t)\,dt, \\
&&\quad v(x_2+0)-v(x_2-0)=-c(x_2-0)\int_0^1(1-t)v(t)\,dt-d(x_2-0)\int_0^1tv(t)\,dt, \\
&&\quad v'(x_1-0)-v'(x_1+0)=-c'(x_1+0)\int_0^1(1-t)v(t)\,dt-d'(x_1+0)\int_0^1tv(t)\,dt, \\
&&\quad v'(x_2+0)-v'(x_2-0)=-c'(x_2-0)\int_0^1(1-t)v(t)\,dt-d'(x_2-0)\int_0^1tv(t)\,dt \bigg\},
\end{eqnarray*}
where
\begin{eqnarray*} 
&&c(x_1+0)=  \frac{2}{\Delta}\Big(1+\frac{x_2^3-x_1^3}{3}-\frac{x_2^2-x_1^2}{2}x_1\Big),\\
&&d(x_1+0)=  -\frac{2}{\Delta}\Big(\big(1+x_2-x_1 -\frac{x_2^2-x_1^2}{2}\big)x_1-\frac{x_2^2-x_1^2}{2}+\frac{x_2^3-x_1^3}{3}\Big),\\
&&c(x_2-0)=  \frac{2}{\Delta}\Big(1+\frac{x_2^3-x_1^3}{3}-\frac{x_2^2-x_1^2}{2}x_2\Big),\\
&&d(x_2-0)=  - \frac{2}{\Delta}\Big(\big(1+x_2-x_1 -\frac{x_2^2-x_1^2}{2}\big)x_2-\frac{x_2^2-x_1^2}{2}+\frac{x_2^3-x_1^3}{3}\Big),\\
&&c'(x_1+0)=-\frac{1}{\Delta} \big(x_2^2-x_1^2\big), \\
&&d'(x_1+0)=\frac{2}{\Delta}\Big(1+x_1-x_2 +\frac{x_2^2-x_1^2}{2}\Big), \\[5pt]
&&c'(x_2-0)=  c'(x_1+0), \;\;\, d'(x_2-0)=d'(x_1+0), \\ [5pt]
&&\Delta=1+x_2-x_1 -\frac{x_2^2-x_1^2}{2}+ \frac{x_2^3-x_1^3}{3}+\frac{x_2-x_1}{12}\big((x_2-x_1)^3+6x_1x_2\big)\neq 0, \\ [5pt]
&&\mbox{since} \;\, x_1, x_2\in(0, 1).
\end{eqnarray*}
And the operator $A_K^*$ acts as follows
$$
A_K^*v=-v''(x)+c''(x)\int_0^1 (1-t)v(t)\,dt+d''(x)\int_0^1 tv(t)\,dt,
$$
where
\begin{align*} 
& c''(x)=\frac{2}{\Delta}\bigg[1+\frac{x_2^3-x_1^3}{3}+\frac{x_2^2-x_1^2}{2}(x_2-x_1)\bigg]\big(\delta'(x-x_1)-\delta'(x-x_2)\big) \\[5pt]
&\qquad\qquad -\frac{1}{\Delta}\big(x_2^2-x_1^2\big)\big(\delta(x-x_1)-\delta(x-x_2)\big),
\end{align*}
\begin{align*} 
& d''(x)= \frac{2}{\Delta}\bigg[\Big(1+x_2-x_1 -\frac{x_2^2-x_1^2}{2}\Big)(x_1-x_2)-\frac{x_2^2-x_1^2}{2}+  \frac{x_2^3-x_1^3}{3}\bigg]\\[6pt]
&\qquad\qquad \times \big(\delta'(x-x_1)-\delta'(x-x_2)\big) \\[5pt]
&\qquad\qquad -\frac{1}{\Delta}\Big(1+x_2-x_1 +\frac{x_2^2-x_1^2}{2}\Big)\big(\delta(x-x_1)-\delta(x-x_2)\big),
\end{align*}
here $\delta$ is the Dirac delta function.




\vskip 1 cm \footnotesize
\begin{flushleft}
   Bazarkan Nuroldinovich Biyarov, \\
   Faculty of Mechanics and Mathematics\\
   L.N. Gumilyov Eurasian National University \\
   13 Munaitpasov St,\\
   010008 Nur-Sultan, Kazakhstan\\
   E-mail: bbiyarov@gmail.com
\end{flushleft} 


\begin{thebibliography}{99.}
\label{bib}
\bibitem{Kokebaev} 
B.K.~Kokebaev,  M.~Otelbaev, A.N.~Shynibekov,  \textit{About expansions and restrictions of operators in Banach space.} Uspekhi Matem. Nauk \textbf{37}  (1982), no. 4, 116--123 (in Russian).  \href{https://clck.ru/TVJYC}{Google Scholar} 

\bibitem{Vishik} 
M.I.~Vishik, \textit{On general boundary problems for elliptic differential equations}. Tr. Mosk. Matem. Obs. 1 (1952), 187--246  (in Russian). English transl.: Am. Math. Soc., Transl., II, 24 (1963), 107--172. \href{https://clck.ru/TVJaL}{Google Scholar}

\bibitem{Biyarov3} 
B.N.~Biyarov,  G.K.~Abdrasheva,  \textit{Relatively bounded perturbations of correct restrictions and extensions of linear operators}.  In: T.Sh.~Kalmenov  et al. (eds.) Functional Analysis in Interdisciplinary Applications. FAIA 2017. Springer Proceedings in Mathematics \& Statistics, vol. 216, 213-221, Springer (2017). \url{https://doi.org/10.1007/978-3-319-67053-9\_20}


\bibitem{Biyarov2}
B.N.~Biyarov,  G.K.~Abdrasheva, \textit{Bounded perturbations of the correct restrictions and extensions}.  AIP Conf. Proc. 1759, (020116). (2016). \url{https://doi.org/10.1063/1.4959730}

\bibitem{Gohberg} 
I.C.\,Gohberg, M.G.\,Krein, \textit{Introduction to the theory of linear nonselfadjoint operators in Hilbert space}. Nauka, Moscow, 1965; Amer. Math. Soc., Providence, 1969. MR  \href{http://www.ams.org/mathscinet-getitem?mr=MR0246142}{0246142} 

\bibitem{Fillmore} 
P.\,Fillmore, J.\,Williams, \textit{On operator ranges}. Advances in Math. \textbf{7},  3 (1971), 254--281. \url{https://doi.org/10.1016/S0001-8708(71)80006-3}

\bibitem{Lions} J.L.\,Lions, E.\,Magenes 
\textit{Non-Homogeneous Boundary Value Problems and Applications I}. Springer-Verlag, Berlin, 1972. \url{https://doi.org/10.1007/978-3-642-65161-8}

\bibitem{Shkalikov} 
A.A.\,Shkalikov, \textit{Basis property of eigenfunctions of ordinary differential operators with integral boundary conditions}. Vestnik Moskov. Univ. Ser. I Mat. Mekh., (1982), no.\,6, 12--21. 
MR  \href{http://www.ams.org/mathscinet-getitem?mr=MR685257}{685257}
Zbl \href{http://zbmath.org/?q=an:0565.34020}{0565.34020}

\bibitem{Shkalikov2} 
A.A.\,Shkalikov, \textit{Perturbations of self-adjoint and normal operators with discrete spectrum}. Russ. Math. Surv. \textbf{71} (2016), no.\,5, 907--964. 
MR  \href{http://www.ams.org/mathscinet-getitem?mr=MR3588930}{3588930}
Zbl \href{http://zbmath.org/?q=an:06691825}{06691825}
\url{https://doi.org/10.1070/RM9740}

\bibitem{Marchenko} 
V.A.\,Marchenko, \textit{Sturm-Liouville operators and their applications}. Naukova Dumka, Kiev, 1977 (in Russian). MR  \href{http://www.ams.org/mathscinet-getitem?mr=MR0481179}{0481179} 








%
\bigskip
%
\end{thebibliography}
\end{document}